\documentclass[11pt, a4paper]{amsart}

\usepackage[pdfusetitle, colorlinks=true]{hyperref}
\usepackage{color}

\usepackage{amsmath, amsthm, amssymb, amsfonts, mathtools}
\usepackage{graphicx,epsfig,color}
\usepackage{exscale}

\theoremstyle{plain}

\newtheorem{theorem}{Theorem}[section]

\newtheorem{lemma}[theorem]{Lemma}

\newtheorem{corollary}[theorem]{Corollary}

\newtheorem{proposition}[theorem]{Proposition}

\theoremstyle{definition}

\newtheorem{definition}[theorem]{Definition}

\theoremstyle{remark}
\newtheorem{remark}[theorem]{Remark}

\newtheorem{example}[theorem]{Example}

\DeclareSymbolFont{AMSb}{U}{msb}{m}{n}
\DeclareMathSymbol{\N}{\mathalpha}{AMSb}{"4E}
\DeclareMathSymbol{\R}{\mathalpha}{AMSb}{"52}
\DeclareMathSymbol{\Z}{\mathalpha}{AMSb}{"5A}
\DeclareMathSymbol{\D}{\mathalpha}{AMSb}{"44}
\DeclareMathSymbol{\s}{\mathalpha}{AMSb}{"53}

\DeclareMathOperator{\vol}{vol}

\DeclareMathOperator{\supp}{supp}

\DeclareMathOperator{\m}{m}
\DeclareMathOperator{\ric}{ric}

\DeclareMathOperator{\id}{id}

\title[Characterization of the null energy condition]{Characterization of the null energy condition via displacement convexity of entropy}
\author{Christian Ketterer}
\address{University of Freiburg}
\email{christian.ketterer@math.uni-freiburg.de}
\thanks{{\it 2010 Mathmatics Subject Classification.} Primary 83C75, 83C57, 49Q22. Keywords: null energy condition, null hypersurface, entropy convexity, singularity theorems.}
\begin{document}
\maketitle
\begin{abstract} We characterize the null energy condition for an $(n+1)$-dimensional, time-oriented Lorentzian manifold in terms of convexity of the  relative $(n-1)$-Renyi entropy along displacement interpolations on null hypersurfaces.  More generally, we also consider Lorentzian manifolds with a smooth weight function and introduce the Bakry-Emery $N$-null energy condition that we characterize in terms of  null displacement convexity of the  relative $N$-Renyi entropy.  As application we then revisit Hawking's area monotonicity theorem for a black hole horizon and the Penrose singularity theorem from the viewpoint of this characterization and in the context of weighted Lorentzian manifolds.
\end{abstract}
\tableofcontents
\section{Introduction}
Many classical theorems in general relativity rest on local geometric constraints for the underlying Lorentzian space-time. These local constraints often have the form of lower bounds for curvature quanitities like the Ricci tensor and via the Einstein equation they are interpreted as energy conditions. One such condition is the null energy condition that requires non-negativity of the Ricci tensor in direction of  null vectors. The null energy condition plays a crucial role in the Penrose Singularity Theorem about incompleteness of null geodesics \cite{penrosenull} which forshadowed the existence of black holes and in Hawking's Area Monotonicity Theorem  \cite{hawkingarea} which asserts that the area of cross-sections of a black hole horizon is non-decreasing towards the future provided the horizon is future null complete.

In this article we present a characterization of the null energy condition for a time-oriented Lorentzian manifold in terms of entropy convexity along the future-directed geodesic null flow on  null hypersurfaces. 

A null hypersurface  $\mathcal H$ in a Lorentzian manifold $(M^{n+1},g)$ is a submanifold of dimension $n$ such that the metric $g$  restricted to $\mathcal H$ degenerates. The entropy is the relative $(n-1)$-Renyi entropy 
$$S_{n-1}(\mu|\vol_{\mathcal H})= -\int \rho^{1-\frac{1}{n-1}} d\vol_{\mathcal H}$$
where $\mu$ is a probability measure on $M$, $\vol_{\mathcal H}$  is the degenerated volume  of the Lorentzian metric $g$ on $\mathcal H$ and $\rho$ is the density of $\mu$ in the Lebesgue decomposition w.r.t. $\vol_{\mathcal H}$. A null flow on $\mathcal H$ consists of null geodesics that foliate the hypersurface $\mathcal H$. We say that a probability measure $\pi$ on $M\times M$ is a null coupling if there exists a null hypersurface $\mathcal H\subset M$ such that $\pi$ is concentrated on 
$$
R_{\scriptscriptstyle\mathcal H} =\big\{ (x,y) \in \mathcal H^2:  \exists \mbox{ flow curve $\gamma_{x,y}$} \mbox{ s.t. } \gamma_{x,y}(s)=x,\ \gamma_{x,y}(t)=y \ \&\  x\leq y\big\}.
$$ Two probability measures $\mu_0, \mu_1\in \mathcal P(\mathcal H, \vol_{\mathcal H})$ are called acausal if they are supported on acausal submanifolds $S_0$ and $S_1$.  Here $\mathcal P(\mathcal H, \vol_{\mathcal H})$ is the set of $\vol_{\mathcal H}$-absolutely continuous probability measures concentrated on $\mathcal H$. We say  that two acausal probability measures  $\mu_0$ and $\mu_1$ are null connected if there exists a null coupling $\pi$ such that the marginal measures of $\pi$ are $\mu_0$ and $\mu_1$.  
For $x,y\in R_{\mathcal H}$ let $t\in [0,1]\mapsto \tilde \gamma_{x,y}(t)$ be the affine reparametrization of $\gamma_{x,y}$.    If $\mu_0$ and $\mu_1$ are null connected, we define $\mu_t= (e_t)_{\#}\pi$ for $t \in [0,1]$ where $e_t:(x,y)\mapsto  \tilde \gamma_{x,y}(t)$ is the evaluation map. We will call $(\mu_t)_{t\in [0,1]}$ null displacement interpolation. 
\subsection{Statement of main result} 
Our main result is the following theorem.
\begin{theorem}\label{mmain1}
Let $(M^{n+1},g)$ be a Lorentzian manifold. 
The null energy condition holds if and only if
 for all $\mu_0, \mu_1\in \mathcal P(M, \vol_{\mathcal H})$ that are acausal and null connected, one has
\begin{align}\label{convexi} S_{n-1}(\mu_t|\vol_{\mathcal H})\leq (1-t) S_{n-1}(\mu_0|\vol_{\mathcal H}) + t S_{n-1}(\mu_1|\vol_{\mathcal H}).\end{align}
\end{theorem}

We  also introduce a Bakry-Emery $N$-null energy condition for weighted Lorentzian manifolds. We say $(M,g,e^{-V})$ with $V\in C^{\infty}(M)$ satisfies the  Bakry-Emery $N$-null condition for $N\geq n-1$  if 
$$
(N'-n+1)\left(\ric(v,v)+ \nabla^2V(v,v)\right) \geq  \langle \nabla V, v\rangle^2
$$
for any null vector $v\in TM$ and any $N'>N$. We obtain the following.
\begin{theorem}\label{mmain2}
Let $(M^{n+1},g, e^{-V})$ be a weighted Lorentzian manifold for $V\in C^{\infty}(M)$. 
The Bakry-Emery $N$-null energy condition holds if and only if
 for all $\mu_0, \mu_1\in \mathcal P(M, \m_{\mathcal H})$ that are acausal and null connected, one has
$$S_{N'}(\mu_t|\m_{\mathcal H})\leq (1-t) S_{N'}(\mu_0|\m_{\mathcal H}) + t S_{N'}(\mu_1|\m_{\mathcal H}) \ \ \forall N'\geq N.$$
where $\m_{\mathcal H}= e^{-V} \vol_{\mathcal H}$. 
\end{theorem}
Theorem \ref{mmain1} and Theorem \ref{mmain2}  are in the   spirit of similar results for the strong energy condition by R. McCann \cite{mccannlorentz} and for  the Einstein equation  by A. Mondino and S. Suhr \cite{mondinosuhr} for which these authors proved charactizations via entropy convexity on the space of probability measures equipped with a Wasserstein distance derived from the time separation function. Such  results are motivated by the desire to formulate  local energy conditions in a low regularity framework where the classical notion of  Ricci curvature breaks down.  In particular, this was the starting point for the development of a  {\it synthetic definition of  the strong energy condition} via entropy convexity \cite{camolorentz} in the framework of Lorentzian length spaces \cite{kusa} (see also \cite{braun, mccannsaemann} for further developments). The characterization of lower Ricci curvature bounds via entropy convexity on the Wasserstein space of probability measures was first established for Riemannian manifolds \cite{cms, sturmrenesse, ottovillani}.
Optimal transport of probability measures that are supported on submanifolds inside of a Riemannian manifold in relation to curvature bounds has been studied in \cite{kemo}. 

We use the characterization of Theorem \ref{mmain2} to define a synthetic $N$-null energy condition for a time-oriented Lorentzian manifold $(M,g)$ equipped with a continuous weight  of the form $e^{-V}$ (Definition \ref{def:syn}).
%
As  an application of our main results we  revisit  Hawking's Area Monotonicity Theorem and the Penrose Singularity Theorem
in the setting of such weighted, time-oriented Lorentzian manifolds that satisfy the synthetic $N$-null energy condition.   First we prove the following version of the monotonicity theorem. 
\begin{theorem}[Hawking Monotonicity] Let $(M^{n+1},g, e^{-V})$ be a  weighted space-time that satisfies the synthetic $N$-null energy condition for $N\geq n-1$ and let $\mathcal H\subset M$ be a null hypersurface.
Assume $\mathcal H$ is future null complete. Let $\Sigma_0, \Sigma_1\subset M$ two acausal spacelike hypersurfaces. Define $\Sigma_i\cap \mathcal H= S_i$, $i=0,1$, and assume $S_0\subset J^-(S_1)$.  Then 
$$\m_{\mathcal H}(S_0)\leq \m_{\mathcal H}(S_1).$$
\end{theorem}
 We also  introduce the notion of synthetically future converging for a codimension 2 submanifold $S$ in a weighted Lorentzian manifold $(M,g,e^{-V})$  (Definition \ref{def:trapped}) and we prove the following version of the Penrose Singularity Theorem.
\begin{theorem}
For $V\in C^0(M)$ let $(M,g,e^{-V})$ be a weighted, time-oriented Lorentzian manifold that is globally hyperbolic. Assume there exists a non-compact Cauchy hypersurface $\Sigma\subset M$, $M$ contains a synthetically future converging, compact, oriented codimension two submanifold $S$ and the synthetic $N$-null energy condition for $N\geq n-1$ holds. Then $M$ is future null incomplete.
\end{theorem}

The rest of this article is structured as follows. 

In Section 2 we will briefly recall important notions about  smooth, time-orientied Lorentzian manifolds and null hypersurfaces. We will define the Ricci and the Bakry-Emery Ricci tensor and the corresponding null energy conditions. We also introduce the degenerated volume and the corresponding relative Renyi entropy. 

In Section 3 we  define the notion of null coupling, null connectedness and null displacement interpolation. We prove that null displacement interpolations are induced by maps. Then, we  prove that the null energy condition is necessary and sufficient for null displacement convexity for the Renyi entropy along null hyupersurfaces. 

In Section 4 we present consequences that we can derive directly from the null displacement convexity. This includes the Hawking area monotonicity and a corollary that leads to the Penrose singularity theorem. 
\subsubsection*{Acknowledgments}
The main part of this work was done when the author stayed at  the Fields Institute in Toronto as  Longterm Visitor during the Thematic Program on {\it Nonsmooth Riemannian and Lorentzian Geometry}. CK wants to thank the Fields Institute for providing  an excellent research environment. In particular many thanks go to Robert McCann, Clemens Saemann and Mathias Braun for stimulating discussions about topics in general relativity.  This work was completed during a research visit of the author
at UNAM Oaxaca, Mexico.
\section{Preliminaries}
\subsection{Lorentzian manifolds}
A space-time is a smooth, connected, time-oriented Lorentzian manifold  $(M,g)$ with $\dim_M=n+1$. The signature of $M$ is $(-, +, \dots, +)$.
We also write $g= \langle \cdot, \cdot \rangle$. A vector $v\in T_xM$ is timelike, spacelike, causal or null if $\langle v, v\rangle$ is negative, positive, non-positive or $0$, respectively.  A $C^1$ curve $\gamma: I \rightarrow M$ is then timelike, spacelike, causal or null if $\gamma'(t)$ is so for every $t\in I$. For $x,y\in M$ we write $x\leq y$ if there exists a future directed causal curve from $x$ to $y$ and  $x\ll y$ if there exists a future directed timelike curve from $x$ to $y$. For $A\subset M$ let  $J^{+}(A):= \{ y\in M:  \exists x\in A \mbox{ s.t. } x\leq y\}$ be the causal future of $A$ and similarly, let $J^-(A):= \{y\in M: \exists x\in A \mbox{ s.t. } y\leq x\}$ be the causal past.   $A$ is called achronal (acausal) if no timelike (causal) curve intersects $A$ twice. An achronal hypersurface $\Sigma$ is a Cauchy hypersurface if every inextendible timelike
curve intersects $\Sigma$ exactly once. A space-time $(M,g)$ which possesses a Cauchy hypersurface,  is called globally hyperbolic.
\subsection{Bakry-Emery Ricci curvature}
\begin{definition}
The Ricci tensor of $(M,g)$ is  $(0,2)$-tensor defined as 
$$
\ric(v,w)=
\mbox{trace}\left[ z\mapsto R(v,z)w\right] = \sum_{i=0}^n \langle R(v, e_i)w, e_i \rangle
$$
where $v,w\in T_xM$ for $x\in M$ and $e_0, e_1, \dots, e_n$ an orthonormal basis of $T_xM$ w.r.t. $g_x$.
If $|\langle v, v\rangle|=1$, one can pick an orthonormal basis $e_0, \dots, e_n$ such that $v=e_j$ for some $j\in \{0, \dots, n\}$ and hence 
$$\ric(v,v)= \sum_{j\neq i=0}^n \langle R(v, e_i)v, e_i\rangle.$$ 
If $v$ is null, there exist orthonormal vectors $e'_2, \dots, e'_n\in v^{\perp}$ such that 
$$\ric(v,v)= \sum_{i=2}^n\langle R(v,e'_i)w, e'_i\rangle$$
\cite[Proof of Lemma 8.9]{oneillsemi}.

Given $V\in C^\infty(M)$ the Bakry-Emery Ricci tensor is 
$$
\ric+ \nabla^2 V=: \ric^V.
$$
Given $N\in [N+1, \infty)$  and $K\in \R$ the weighted space-time $(M,g, e^{-V})$ satisfies the  timelike Bakry-Emery $N$-Ricci curvature condition if 
$$
(N'-n+1)(\ric^V(v,v)- K\langle v, v\rangle)\geq \langle \nabla V, v \rangle^2
$$
for any timelike vector $v\in TM \mbox{ and } \forall N'>N.$
In particular, if $N=n+1$, it follows that $V$ is constant along any timelike curve.  A characterization of the timelike Bakry-Emery $N$-Ricci curvature condition in term of entropy displacement convexity was given in \cite{mccannlorentz, mondinosuhr}.
\end{definition}
\begin{definition}\label{def:benull}
Let $N\in [n-1, \infty)$. We say $(M,g,e^{-V})$ satisfies the  Bakry-Emery $N$-null condition  if 
$$
(N'-n+1)\ric^V(v,v) \geq  \langle \nabla V, v\rangle^2
$$
for any null vector $v\in TM$ and any $N'>N$.

If $N=n-1$, it follows that $V$ is constant along any lightlike curve. Hence $\langle \nabla V, v\rangle=\nabla V^2(v,v)=0$ for every null vector $v$ and the condition reduces to  the  null energy condition
$$\ric(v,v)\geq 0 \ \mbox{ for any null vector } v\in TM.$$
\end{definition}
\subsection{Null hypersurfaces}\label{subsec:null}
A null hypersurface $\mathcal H$ in $M$ is an embedded $C^2$ submanifold of codimension one such that the pull-back  metric of $g$ on $\mathcal H$ is degenerated.  There exists a non-vanishing $C^1$ future-directed, null vector field $K\in \Gamma(T\mathcal H)$ such that $K_x^{\perp}= T_x\mathcal H$ for all $x\in \mathcal H$. 
The flow curves of $K$ in $\mathcal H$ are null geodesics and one says that $\mathcal H$ is null geodesically generated (ruled).  A flow curve admits an affine reparametrization. More precisely, let $\gamma_x: (\alpha, \beta)\rightarrow \mathcal H$ be the flow curve of $K$ with $\gamma_x(0)=x$. There exists a reparametrization $\varphi: (a,b)\rightarrow (\alpha, \beta)$ of $\gamma_x$ such that $ \gamma_x\circ \varphi(t)= \exp_x(tK(x))$. The vectorfield $K$ is unique up to a positive factor.  For further details about the geometry on null hypersurface we refer to \cite[Appendix A]{chdegaho}.
\begin{example}\label{ex:nullcones}
Let $S$ be a codimension 2 smooth hypersurface that is spacelike. Assume also that $S$ is an oriented manifold. At any point $x\in S$ we can pick two future-directed orthogonal null vectors $L_x$ and $\underline L_x$. Assume $L_x$ projects to the exterior of $S$. We will call $L_x$ an outer null normal to $S$ at $x$ and $\underline L_x$ an inner null normal to $S$ at $x$. We choose  the  map $x\in S\mapsto L_x$ to be smooth. For every $x\in S$ there exists a null geodesic $\gamma_x(t)= \exp_x(tL_x)$ with $\gamma_x(0)$ and $\gamma_x'(0)=L_x$ and we consider the set $\mathcal C= \bigcup_{x\in S}\mbox{Im} \gamma_x$ formed by these geodesics. $\mathcal C$ is also called a null geodesic congruence.  By replacing $L$ with $\underline L$ we also define $\underline {\mathcal C}$.  A neighborhood $\mathcal H$ of $S$  in $\mathcal C$ is  a smooth submanifold of codimension two and  the vector field $L$ extends to a smooth vector field  $K$ on $\mathcal C$ via $K(\gamma_x(t))= \gamma_x'(t)$.  Then $\mathcal H$ is a null hypersurface. Similarly there exists a corresponding null hypersurface $\underline{\mathcal H}$ for $\underline{\mathcal C}$.
\end{example}

\subsubsection{Volume}\label{sec:vol}
Let  $\mathcal H\subset M$ be a null hypersurface.
The Lorentzian metric $g$ induces a $(n-1)$-form $\vol_{\mathcal H}$ on $\mathcal H$. 
If $S\subset \mathcal H$ is spacelike, codimension $2$ $C^2$ submanifold, then $\vol_{\mathcal H}(A)= \int_A d\vol_{\mathcal H}$ whenever $A\subset S$  is measurable.  If $V\in C^{\infty}(M)$, we set $\m_\mathcal H= e^{-V} \vol_{\mathcal H}$. 

A probability measure $\mu\in \mathcal P(M)$ is called $\m_{\mathcal H}$-absolutely continous if $\mu= \rho d\m_{\mathcal H}$ for a measurable function $\rho: M\rightarrow [0,\infty)$. We write $\mathcal P(\mathcal H ,\m_{\mathcal H})$ for the  set of $\m_{\mathcal H}$-absolutely continuous probability measures.
\subsubsection{Relative entropy}
The relative $N$-Renyi entropy w.r.t. $\m_{\mathcal H}$ of a probability $\mu\in \mathcal P(M)$ is defined as 
\begin{align*}
\mu= \rho d\m_{\mathcal H} + \mu^s\mapsto S_N(\mu|\m_{\mathcal H})= - \int \rho^{1-\frac{1}{N}} d\m_{\mathcal H}
\end{align*}
where $\rho d\m_{\mathcal H} + \mu^s$ is the Lebesgue decomposition of $\mu$.
Basic properties of $S_N$ are $S_N(\mu|\m_{\mathcal H})\leq 0$ with $"="$ if $\mu\perp \m_{\mathcal H}$ and $S_N(\mu|\mu_{S})\geq - \m_{\mathcal H}(\supp \mu)^{\frac{1}{N}}$. 
\section{Proof of the main results}
\subsection{Null couplings}
\begin{definition} Let $(M,g)$ be a space-time and $\mathcal H\subset M$ a null hypersurface with tangent vector field $K$ that generates a geodesic null flow.
We define a transport relation $R_{\mathcal H}$ on $\mathcal H$ by
$$
R_{\scriptscriptstyle\mathcal H} =\big\{ (x,y) \in \mathcal H^2:  \exists \mbox{ flow curve $\gamma$ of } K \mbox{ s.t. } \gamma(s)=x,\ \gamma(t)=y \ \&\  x\leq y\big\}.
$$
Let $x,y \in R_{\scriptscriptstyle\mathcal H}$, $\gamma|_{[s,t]}|=: \gamma_{x,y}$ the flow curve that connects $x$ and $y$, and let  $\tilde \gamma_{x,y}:[0,1]\rightarrow \mathcal H$ its affine reparametrization.  Let $\mathcal G(\mathcal H):= \mathcal G^{[0,1]}(\mathcal H)$ be the collection of all affine parametrized flow curves  $\tilde \gamma:[0,1]\rightarrow \mathcal H$, and for $\tilde \gamma^0\neq \tilde \gamma^1\in \mathcal G$ we define 
$$d_{\infty}(\tilde \gamma^0, \tilde \gamma^1)= \sup_{t\in [0,1]} d_g(\tilde \gamma^0(t), \tilde \gamma^1(t)). $$
This induces a topology on $\mathcal G(\mathcal H)$. 

A probability measure $\pi\in \mathcal P(M^2)$ is called a {\it null coupling} if there exists a null hypersurface $\mathcal H$ such that $\supp \pi \subset R_{\scriptscriptstyle \mathcal H}$.  By measurable selection we can define a measurable map $\Upsilon: (x,y)\mapsto \mathcal G(\mathcal H)$  $\pi$-almost everywhere such that $\Upsilon(x,y)=\tilde \gamma$ if and only if $\tilde \gamma(0)=x$ and $\tilde \gamma(1)=y$.  We call the pushforward $\Upsilon_{\#} \pi = \Pi$ dynamical (null) coupling.  If $e_t: \mathcal G(\mathcal H)\rightarrow \mathcal H$, then we call $\mu_t=(e_t)_{\#} \Pi, t\in [0,1],$ {\it null displacement interpolation}. 

Let  $(\mu_t)_{t\in [0,1]}$ be a null displacement interpolation. For $\mu_t$-almost every $x\in \mathcal H$ there exists a unique $\tilde \gamma^t_{x}\in \mathcal G(\mathcal H)$ such that $\tilde \gamma^t_{x}(t)=x$. Then we also have the $\mu_t$-almost everywhere defined  map $\Lambda_t: x\mapsto \tilde \gamma^t_{x}$.  One can check that  $(\Lambda_t)_{\#}\mu_t= \Pi$.

We call a probability measures $\mu$ {\it acausal} if there is a  space-like acausal   $C^2$ submanifold $\Sigma$ such that $\mu$ is concentrated in $\Sigma$.  

We call two acausal probability measures $\mu_0$ and $\mu_1$ {\it null connected} if there exists a coupling $\pi$ between $\mu_0$ and $\mu_1$ that is null. In particular, it holds $\supp \mu_i\subset \mathcal H\cap \Sigma_i$, $i=0,1$, for a null hypersurface $\mathcal H$ and two spacelike, acausal submanifolds $\Sigma_0$ and $\Sigma_1$.
\end{definition}

\begin{lemma}\label{lem:optimalmap} A null coupling $\pi$ between acausal probability measures $\mu_0$ and $\mu_1$ as above is induced by a $C^1$ map $T: U\rightarrow \mathcal H$ where
$U\subset S_0$ open such that $\supp\mu_0\subset U$. More precisely, there exists $r\in C^1(U)$ such that $T(x)= \exp_x(r(x)K(x))$ and $(\id_{{\supp}(\mu_0)}, T)_{\#}\mu_0$.
We call $T$ transport map.
\end{lemma}
\begin{proof} {\bf 1.} Since $\Sigma_1$ is acausal, 
for every $x\in \Sigma_0$ there exists a most one $y\in \Sigma_1$ such that $(x,y)\in R_{\scriptscriptstyle \mathcal H}$. Moreover, by existence of a null coupling $\pi$  for $\mu_0$-almost every $x$ there exists at least one such $y\in \Sigma_1$.  Hence $y$ is the unique intersection of $\gamma_x$ with $\Sigma_1$. So we set $\mu_0$-almost everywhere $x\mapsto y=: \tilde T(x)$, and for $\pi$-almost every $(x,y)$ it holds $\tilde T(x)=y$. 
\smallskip\\
{\bf 2.}
Let $\Phi:   S_0\times (0, \infty)\rightarrow M$  be the map given by $\Phi(x,t)= \exp_x(tK(x))$. 
For every $(x,t)\in S_0\times (0,\infty)$ with $\Phi(x,t)\in \mathcal H$ there is exactly one inextendible flow curve $\gamma_y$ of $K$ in $\mathcal H$ passing through $\Phi(x,t)$.  Hence $\Phi(x,s) = \gamma_y\circ \varphi_x(s)$ for a regular  reparametrization $\varphi_x: (0, \omega_x)\rightarrow \R$. In particular $\Phi(x,s)\in \mathcal H$ for $s\in (0, \omega_x)$.
Hence, by Lemma 4.15 in \cite{chdegaho}  there are no focal points along $t\in (0, \omega_x)\mapsto \Phi(x,t)$.
Consequently $D\Phi_{(x,t)}: T_{x,t}(S_0\times (0,\infty)) \rightarrow T_{\Phi(x,t)} M$ is injective for $t\in (0, \omega_x)$.
\smallskip\\
{\bf 3.}
If $y\in S_1$ with $\Phi(x,t)=y$ for $x\in S_0$, we can find an open set $O\times (t-\delta, t+\delta)\ni (x,t)$ in $S_0\times (0,\infty)$ such that $\Phi|_{O\times (t-\delta, t+\delta)}$ is a diffeomorphism.   Since fore every $x\in O$ the curve $\Phi(x,t)$ meets $S_1$ at most once and since $\Phi(O\times (t-\delta, t+\delta))\cap S_1$ is open in $S_1$, we can choose $O$ small enough such that for every $z\in O$ there exists a unique $s\in (t-\delta, t+\delta)$ such that  $\Phi(z,s)\in S_1$. Then the implicite function theorem implies existence of a $C^1$ function $g:  O\rightarrow (t-\delta, t+\delta)$ such that 
$$\exp_x(g(x)K(x))=\Phi(x, g(x))\in S_1\ \ \forall x\in O\subset  S_1.$$
Since $\tilde T(x)$ is already the unique intersection point of the flow curve $\gamma_x$ with $S_1$, we have $\tilde T(x)= \exp(g(x)K(x))$ for all $x\in \supp\mu_0\cap O$. Hence, the map $\tilde T$ is the restriction of the $C^1$ map $T: x\in U\mapsto \exp_x(r(x)K(x))$ with $r: U\rightarrow (0,\infty)$ for $U\subset S_0$ open. 
\end{proof} 
\begin{example}\label{example} Consider a space-time $(M,g)$ and let $\Sigma_0, \Sigma_1\subset M$ be spacelike, acausal submanifolds and let $S_i= \mathcal H \cap \Sigma_i$ such that $S_0\subset J^-(S_1)$. Since $\Sigma_1$ is acausal, for every $x\in S_0$ there exists at most one $y\in S_1$ such that $J^-(y)\ni x$. Then, exactly as in the proof of the previous lemma we can construct a  $C^1$ map $T: S_0 \rightarrow S_1$ of the form $T(x)= \exp_x(\tilde K(x))$. In particular, if $\vol_{\mathcal H}(S_0)<\infty$ (for the definition of $\vol_{\mathcal H}$ see Section \ref{sec:vol} below), we can define $\mu_0= \frac{1}{\vol_{\mathcal H}(S_0)}\vol_{\mathcal H}|_{S_0}$ and $\mu_1= (T)_{\#} \mu_0$. Then $\mu_0$ and $\mu_1$ are null connected. 
\end{example}
\begin{remark}  
The proposition implies that $\exp_x(t\tilde K(x))= \tilde \gamma_{x, T(x)}(t)=:\tilde \gamma_x(t)$ where $\tilde K(x)= r(x) K(x)$ and $\tilde \gamma_{x,T(x)}= \Upsilon(x,T(x))$. 
We set $T_t: U\rightarrow \mathcal H$ via $T_t(x):= \exp_x(t \tilde K(x))=\tilde \gamma_x(t)$.  In particular, we have for the induced null displacement interpolation that $\mu_t= (T_t)_{\#}\mu_0$. Let $\omega_x>1$ such that $\tilde \gamma_x([0,\omega_x)) \subset \mathcal H$ and let $b\in [0, \omega_x)$. 
For  $x\in U$ and  $v\in T_xU$ the vectorfield $t\in [0,b]\mapsto (DT_t)_xv$ satisfies the Jacobi equation with $DT_0v=v$ and $\frac{d}{dt}DT_{t,x}v= \nabla_v \tilde K(x)$.  Since, by Lemma 4.15 in \cite{chdegaho}, there are no focal points along $\tilde \gamma_x(t)$ unless there is an endpoint of $\gamma_x$ in $\mathcal H$,  $DT_{t,x}v$ is not lightlike for any $v\in T_xS_0$. (Since $\mathcal H^2$ is $C^2$ submanifold and $K$ a non-vanishing $C^1$ vectorfield on $\mathcal H$, there are no endpoints in $\mathcal H$.)  Hence $T_t$ is a $C^1$ diffeomorphism with $T_t(S_0)=S_t$ spacelike and acausal.  In particular $g|_{S_t}$ is non-degenerated.

\end{remark}
\begin{remark}
This setup is also meaningful in a more general context. 
A general null hypersurface $\mathcal H$ is a topological codimension $1$ submanifold that is the union of null geodesics. The concept of transport relation and null coupling are defined analogously. A measure $\mu$ is acausal if it is supported on a $n-1$-rectifiable subset that is acausal. The definition of $r$ and the map $T$ is the same and by \cite{chdegaho} $T$ and $r$ are Lipschitz continuous.  
\end{remark}
\subsection{The null energy condition implies null displacement convexity}

Let $\mu_0, \mu_1\in \mathcal P(M, \m_{\mathcal H})$ be  acausal and null connected and let $(\mu_t)_{t\in [0,1]}$ be the induced null displacement interpolation. Each $\mu_t$ is supported on the set $T_t(U)\subset S_t$ where $S_t$ is the image of $S_0$ under $T_t: U\rightarrow \mathcal H$.
\begin{lemma} Let $V\in C^0(M)$ and $\m_{\mathcal H}=e^{-V}\vol_{\mathcal H}$. Then
$\mu_t\in \mathcal P(M, \m_{\mathcal H})$ $\forall t\in (0,1)$.  
We denote the densities of $\mu_t$ w.r.t. $\m_{\mathcal H}$ with $\rho_t$. 
\end{lemma}
\begin{proof}
Let $N\subset S_t$ with $\m_{S_t}(N)=0$.  $T_t$ is a $C^1$ diffeomorphism and hence $T_t^{-1}$ is a $C^1$. It follows $(T_t)^{-1}(N)$ has $0$ measure w.r.t. $\vol_{\mathcal H}$ and therefore w.r.t. $\m_{\mathcal H}$. Let $\Pi$ be the dynamical null coupling and $\pi= (e_0, e_t)_{\#}\Pi$. Then
$$\mu_t(N)= \pi(S_0, N)= \pi((T_t)^{-1}(N), N)\leq \pi((T_t)^{-1}(N), S_t)= \mu_0((T_t)^{-1}(N)).$$
Hence $\mu_t(N)=0$ and $\mu_t$ is $\m_{\mathcal H}$ absolutely continuous. 
\end{proof}
\begin{lemma}\label{lem:ma}  Let $V\in C^0(M)$ and $\m_{\mathcal H}=e^{-V}\vol_{\mathcal H}$.
Let $\mu_t = \rho_t \m_{\mathcal H}$. Then 
$$e^{-V\circ T_t}\rho_t(T_t(x)) \det DT_{t}(x)= \rho_0(x)e^{-V(x)}.$$
\end{lemma}
\begin{proof}
Let $W\subset S_0$ be an measurable and arbitrary. $S_0$ and $S_t$ are equipped with the restricted metric $g$. Then 
\begin{align*}
\mu_0(W)&= \mu_t(T_t(W))= \int_{T_t(W)} d\mu_t\\
&= \int_{T_t(W)} e^{-V} \rho_t d\vol_{\mathcal H}= \int_We^{-V\circ T_t} \rho_t\circ T_t \det DT_{t} d\vol_{\mathcal H}
\end{align*}
where we used the  area formula for the map $T_t: S_0\rightarrow S_t$. Since $W$ was arbitrary, the claim follows. 
\end{proof}
\begin{theorem}\label{mainforw} Let $(M, g)$ be  a smooth, time-oriented, $(n+1)$-dimensional Lorentzian manifold and $V\in C^{\infty}(M)$. Let $\mu_0, \mu_1\in \mathcal P(M)$ be acasual  and  null connected, and let $(\mu_t)_{t\in [0,1]}$ be corresponding null displacement interpolation.

\begin{enumerate}
\item If  the null energy condition holds and $\mu_0, \mu_1\in \mathcal P(M, \vol_{\mathcal H})$, then it follows  that
$t\in [0,1]\mapsto S_{n-1}(\mu_t|\vol_{\mathcal H})$ is convex.
\smallskip
\item  If the weighted Lorentzian manifold $(M,g, e^{-V})$ satisfies the Bakry-Emery $N$-null energy condition for $N\geq n-1$ and $\mu_0, \mu_1\in \mathcal P(M, \m_{\mathcal H})$, then it holds that
$t\in [0,1]\mapsto S_{N'}(\mu_t|\m_{\mathcal H})$ $\forall N'\geq N$. 
is convex.
\end{enumerate}
\end{theorem}
\begin{proof} Let $\tilde K(x)= r(x) K(x), x\in U\subset S_0,$ the vector field given by Lemma \ref{lem:optimalmap} and let $T_t(x)=\exp_x(t\tilde K(x))=\tilde \gamma_x(t)$. The family of maps
$$t\in [0,1]\mapsto A_x(t):={DT}_{t}(x)
$$
from $T_xS_0$ to $T_{\tilde \gamma_x(t)}S_t$
satisfies the Jacobi equation 
\begin{align}\label{jacobi} A''_x+ R(A_x, \tilde \gamma'_x)\tilde\gamma'_x=0
\end{align}
with $A_x(0) v=  v$ and $A_x'(0) v={{ \nabla}_{ v} \tilde K}$ where $R$ is the curvature tensor.  
\smallskip

We define $$t\in [0,1]\mapsto U_x(t) =  A'_x(t)  A_x^{-1}: T_{\tilde \gamma_x(t)}S_t\rightarrow T_{\gamma_x(t)}\mathcal H$$ and check that it satifies the Riccatti equation
\begin{align*}U_x'(t)= A_x''(t) A^{-1}_x(t)- A_x'(t)A^{-1}_x(t)A_x'(t)A_x^{-1}(t)= -{ R}(\cdot, \gamma_x')\gamma_x'- U_x^2(t).
\end{align*}
This equation is also known as the {\it Optical Equation} \cite{chdegaho}.
\smallskip

Let $E_1(t), \dots,  E_{n-1}(t)\in T_{\tilde \gamma_x(t)}S_t$ be the solutions of the linear equation $$P_t\nabla_{\tilde \gamma_x'(t)} E_i=0, \ \ E_i(0)=e_i$$ for an orthonormal Basis $e_1, \dots, e_{n-1}$ of 
$T_{x}S_0$.  
The map $P_t: T_{\tilde\gamma_x(t)}\mathcal H\rightarrow T_{\tilde \gamma_x(t)}S_t$ is the linear projection. 
\smallskip

{\it Claim.} $E_1(t), \dots, E_{n-1}(t)$ is an orthonormal basis of $T_{\tilde \gamma_x(t)}S_t$. 
\smallskip

\noindent 
{\it Proof of the claim.} We observe $$\frac{d}{dt}\langle E_i(t), E_j(t)\rangle  = \langle E_i'(t), E_j(t)\rangle + \langle E_i(t), E_j'(t)\rangle =0.$$
Hence $\langle E_i(t), E_j(t)\rangle$ is constant along $\tilde \gamma_x$. $ \blacktriangle$

The trace of $U_x(t)$ restricted to $T_{\tilde \gamma_x(t)}S_t$ is then 
$$
\mbox{tr} U_x(t)= \sum_{i=1}^{n-1} \langle E_i(t), U_x(t) E_i(t)\rangle = \sum_{i=1}^{n-1} \langle E_i, P_tU_x(t)E_i(t)\rangle.
$$

{\it Claim.} The derivative $P_t'$ is a linear map from $T_{\tilde \gamma_x(t)}\mathcal H\rightarrow T_{\tilde \gamma_x(t)}S_t^{\perp}$. 
\smallskip

\noindent
{\it Proof of the claim.}   $P'_t: T_{\tilde \gamma_x(t)}\mathcal H\rightarrow T_{\tilde \gamma_x(t)}\mathcal H$ is a linear map.  
Note that $$E_1(t), \dots, E_{n-1}(t), \tilde \gamma'_x(t)= E_n(t)$$ is an orthogonal basis of $T_{\tilde \gamma_x(t)}\mathcal H$
and
 $\langle P_tE_i(t), E_j(t)\rangle = \langle E_i(t), E_j(t)\rangle$ for $i,j=1, \dots, n$. Therefore $$0= \frac{d}{dt}\langle P_t E_i(t), E_j(t)\rangle= \langle \left(\frac{d}{dt} P_t\right) E_i(t), E_j(t)\rangle \ \ \forall i,j\in \{0,\dots, n\}.$$
This proves the claim.\ \ $ \blacktriangle$
\smallskip

{\it Claim.}  $\mbox{tr}(U_x)'= \mbox{tr}(U_x')$. 
\smallskip\\
{\it Proof of the claim.}
We compute  
\begin{align*}
\langle E_i, U_xE_i\rangle'& = \langle E_i, P U_xE_i\rangle'\\
&= \langle E_i', PU_xE_i\rangle + \langle E_i, P' U_xE_i\rangle + \langle E_i, PU_x'E_i\rangle + \langle E_i, PU_x PE_i'\rangle\\
&= \langle E_i, PU_x'E_i\rangle.
\end{align*}
Summing w.r.t. $i=1, \dots, n-1$ yields the claim. \  $ \blacktriangle$
\smallskip

Hence taking the trace of the optical equation we obtain
$$\mbox{tr}(U_x)' + \mbox{tr}(U_x^2) + \mbox{tr} R(\cdot, \tilde\gamma_x')\tilde \gamma_x'=0$$
where $\mbox{tr} R(\cdot, \tilde \gamma_x')\tilde \gamma_x= \sum_{i=1}^{n-1}\langle R(\bar E_i, \tilde \gamma_x')\tilde \gamma_x', \bar E_i\rangle= \ric(\tilde \gamma_x', \tilde \gamma_x')\geq 0$ and the last inequality follows from the null energy condition.
\smallskip

{\it Claim.} Fix $t_0$. $U_x(t_0)$ is a self-adjoint operator on $T_{\tilde \gamma_x(t_0)}S_{t_0}$.
\smallskip\\
{\it Proof of Claim.} Note that $U_x(t) = A_x'(t) A^{-1}_x(t)$ and $A_x(t_0+s)z=J(s)$ satisfies the Jacobi equation with $J(0)=v\in T_{\tilde \gamma_x}S_{t_0}$ where $vA^{-1}_x(t_0)=z$ and $J'(0)=w\in T_{\tilde \gamma_x(t_0)}\mathcal H$.  By standard Riemannian calculus we can write 
$$J(s)= \frac{\partial}{\partial \epsilon} \Big|_{\epsilon=0}\exp_{\alpha(\epsilon)} (s \hat K\circ \alpha(\epsilon))$$
where $\alpha:(-\epsilon, \epsilon)\rightarrow \mathcal H$ is $C^2$ and satisfies $\alpha(0)=\tilde \gamma_x(t_0)=:y$,  $\alpha'(0)=v$  and $\hat K$ is a $C^1$ vectorfield on $\mathcal H$ in a neighborhood of $y$  such that $\hat K(0)= \tilde \gamma_x'(t_0)$ and $\nabla_v \hat K(y)= w$.  Hence, $U_x(t_0)v = A_x'(t_0)z=J'(0)=w= \nabla_v\hat K(y)$. Since $\tilde K$ is also a normal vectorfield, this is the null Weingarten map  of $\mathcal H$ in $y=\tilde \gamma_x(t_0)$. Hence $U_x(t_0)$ is self-adjoint. \ $ \blacktriangle$
\smallskip

By an application of the Cauchy-Schwarz inequality it follows  $$\mbox{tr}(U^2_x)\geq \frac{1}{n-1} \mbox{tr}(U_x)^2.$$
Hence
$\mbox{tr}(U_x)' + \frac{1}{n-1} \mbox{tr}(U_x)^2 \leq 0.$ 
\smallskip

We set $ \det(A_x)=y_x$. It follows  that 
$(\log y_x)'= \mbox{tr}(A_x' A_x^{-1})= \mbox{tr}(U_x)$ and therefore 
\begin{align}\label{rara}(\log y_x)''+ \frac{1}{n-1}((\log y_x)')^2\leq 0\end{align}
as well 
$\frac{d^2}{dt^2} y_x^{\frac{1}{n-1}}\leq 0$, or equivalently
$$y_x^{\frac{1}{n-1}}(t)\geq (1-t)\underbrace{y_x^{\frac{1}{n-1}}(0)}_{=1} + t\underbrace{y^{\frac{1}{n-1}}_x(1)}_{=(\det DT(x))^{\frac{1}{n-1}}}.$$
Equation \eqref{rara} also follows from the {\it Raychaudhuri equation}. However for this paper we gave a derivation of \eqref{rara} that is closer to ideas in optimal transport.
\smallskip

We have $\mu_t= (T_t)_{\#} \mu_0$ with $T_t$. Then,  together with Lemma \ref{lem:ma},
\begin{align*}
S_{n-1}(\mu_t | \vol_{\mathcal H})&= -\int \rho_t^{1-\frac{1}{n-1}} d\vol_{\mathcal H}\\
&= -\int \rho_t^{-\frac{1}{n-1}} \circ T_t d\mu_0\\
&= - \int_{S_0} \rho_0^{-\frac{1}{n-1}} (\det DT_t)^{\frac{1}{n-1}} d\mu_0\\
&\leq - (1-t) \int \rho_0^{-\frac{1}{n-1}} d\mu_0 - t\int \rho_0^{-\frac{1}{n-1}}( \det DT)^{\frac{1}{n-1}} d\mu_0\\
& ={ (1-t) S_{n-1}(\mu_0|\m_{\mathcal H}) + t S_{n-1}(\mu_1|\m_{\mathcal H}).}^{}
\end{align*}
This finishes the proof of (1).

Similarly, we can consider $z_x(t) =e^{- {V\circ \tilde \gamma_x(t)}}y_x(t)$. It follows that 
\begin{align*}&(\log z_x)''(t)=-\langle \nabla^2 V|_{ \tilde \gamma_x(t)} \tilde \gamma_x'(t) , \tilde \gamma_x'(t)\rangle  + (\log y_x)''(t)\\&\leq -\ric^V(\tilde \gamma_x(t), \tilde \gamma_x(t))- \frac{1}{n-1}((\log y_x)')^2  - \frac{1}{N'-n+1}\langle V\circ \tilde \gamma_x(t), \tilde \gamma_x'(t)\rangle^2\\
&\leq- \frac{1}{N'} ((\log z_x)')^2
\end{align*}
where we used $\frac{1}{n-1} a^2+ \frac{1}{N'-(n-1)}b^2 \geq \frac{1}{N'} (a+b)^2$ and the Bakry-Emery $N$-null energy condition.

Then we also have $\frac{d^2}{dt^2} z_x^{\frac{1}{N'}}\leq 0$. Again using Lemma \ref{lem:ma}, it follows
\begin{align*}
S_{N'}(\mu_t | \m_{\mathcal H})&= -\int \rho_t^{1-\frac{1}{N'}} d\m_{\mathcal H}= -\int \rho_t^{-\frac{1}{N'}} \circ T_t d\mu_0\\
&= - \int_{S_0} \rho_0^{-\frac{1}{N'}} e^{-\frac{1}{N'}( V\circ T_t-V)}\det DT_t^{\frac{1}{N'}} d\mu_0\\
&\leq - (1-t) \int \rho_0^{-\frac{1}{N'}} d\mu_0 - t\int \rho_0^{-\frac{1}{N'}} e^{-\frac{1}{N'}(V\circ T_1 - V) }\det DT_t^{\frac{1}{N'}} d\mu_0\\
& ={ (1-t) S_{N'}(\mu_0|\m_{\mathcal H}) + t S_{N'}(\mu_1|\m_{\mathcal H}).}^{}
\end{align*}
This finishes the proof of (2).
\end{proof}
\subsection{Null displacement convexity implies the null energy condition}
\begin{theorem}\hspace{0mm}
\begin{enumerate}
\item 
Let $(M,g)$ be a $(n+1)$-dimensional, time-oriented  Lorentzian manifold. 
Assume for every null hypersurface $\mathcal H$ and for every $\mu_0, \mu_1\in \mathcal P(M, \vol_{\mathcal H})$ that are null connected via a null coupling $\pi$, it holds 
$$S_{n-1}(\mu_t|\vol_{\mathcal H})\leq (1-t) S_{n-1}(\mu_0|\vol_{\mathcal H}) + t S_{n-1}(\mu_1|\vol_{\mathcal H})$$
where $\mu_t$ is the induced null displacement interpolation. 
\smallskip\\
Then $(M, g)$ satisfies the null enery condition. 
\smallskip
\item Let $(M,g, e^{-V})$ be a weighted, $(n+1)$-dimensional, time-oriented Lorentzian manifold with $V\in C^{\infty}(M)$ and assume for every null hypersurface $\mathcal H$ and for every $\mu_0, \mu_1\in \mathcal P(M, \m_{\mathcal H})$ that are null connected via a null coupling $\pi$, it holds $\forall N'> N\geq n-1$ that
$$S_{N'}(\mu_t|\m_{\mathcal H})\leq (1-t) S_{N'}(\mu_0|\m_{\mathcal H}) + t S_{N'}(\mu_1|\m_{\mathcal H}) ,$$ 
Then $(M, g, e^{-V})$ satisfies the Bakry-Emery $N$-null  energy condition.
\end{enumerate}
\end{theorem}
\begin{proof} We argue by contradiction in both cases.  
\smallskip\\ 
For (1) assume there exists a null vector $v\in TM$ and  $\varepsilon>0$ such that 
$$\ric(v,v) \leq - 6 \varepsilon  <0.$$
For (2) assume $\exists$  $N'>N$ and a null vector $v\in T_pM$ with $$(N'-n+1)\ric^V(v,v)- \langle \nabla V, v\rangle^2<0$$ and let $\varepsilon >0$ such that $$\ric^V(v,v)-\frac{1}{N'-n+1}\langle \nabla V, v\rangle^2 \leq -6\varepsilon  <0.$$
In the following by a slight abuse of notation we write $N'=N$.

{\bf 1.}
Consider the exponential map $\exp_p: \mathcal U\subset T_pM\rightarrow M$  in $p$ where $\mathcal U$ is an open subset of $T_pM$ such that $\exp_p$ is a diffeomorphismus. Let $\mathcal I: \R^{n+1}_1\rightarrow (T_pM, g_p)$ an isometrie between the Minkowski space $\R^{n+1}_1$ and the tangent space at $p$. Hence $\phi:= \exp_p \circ \mathcal I: \mathcal I^{-1}(\mathcal U)=:\mathcal V\rightarrow \phi(\mathcal V)$ is a normal coordinate map. 

We choose an orthonormal basis $e_0, e_1, \dots, e_n$  in $\R^{n+1}_1$ with $e_0$ timelike and $e_1, \dots, e_n$ spacelike  such that  $e_0+e_1= \sigma v$ for some $\sigma>0$.  
The vectors $e_1,  \dots, e_n$ span a spacelike linear subspace $\Lambda'= \langle e_1, e_2, \dots, e_n\rangle$ in $\R^{n+1}_1$ that is isometric to $\R^n$. 

In $\Lambda'\simeq \R^n$ we choose a codimension 1 submanifold $\mathcal S$ such that $0\in \mathcal S$ and such that the second fundamental form in $0$ satisfies $\Pi_p= \lambda \langle \cdot, \cdot  \rangle_{\R^{n-1}}$ for $\lambda\in \R$ that we choose later. 

We define $\Sigma= \phi(\Lambda'\cap B_\delta(0))$ and $S= \phi(\mathcal S\cap B_\delta(0))$ where $B_\delta(0)$ is the ball w.r.t. the standard Euclidean metric for some $\delta>0$ such that $B_\delta(0)\subset \mathcal I^{-1}(\mathcal U)$. $\Sigma$ and $S$ are spacelike submanifolds in $M$ such that $S\subset \Sigma$.

Let $N_0: S\rightarrow T\Sigma$ be the smooth normal unit vectorfield of $S$ as submanifold in $\Sigma$, and let $N_1: \Sigma\rightarrow TM$ be the future pointing normal unit vector field on $\Sigma$ in $M$ such that we have $D\phi_0(e_1)= N_0(p)$ and $D\phi_0(e_0)=N_1(0)$.  Hence $\sigma v=  N_0(p) + N_1(p)$. We define $N: S\rightarrow TM$ as $N_0+ N_1|_{S}$. It follows that $N$ is a null vector field along $S$ and 
$$\nabla_v N= \nabla_vN_0 \ \ \forall v\in T_pS$$
since the covariant derivative of $N_1$ at $p$ in direction of vectors in $T_p\Sigma$  vanishes.  Here we use the properties of normal coordinates at the base point. Hence, we also have 
$$\langle \nabla_v N(p), v\rangle = \Pi^{\Sigma}_S(v,v)$$
where $\Pi^{\Sigma}_S$ is the second fundamental form of $S$ in the ambient space $\Sigma$.

We can extend $N: S\rightarrow TS$ to a smooth null vector field on $\Sigma$ that we  denote as $\hat N$. 
We define the map $\Phi: \Sigma\times (-\epsilon, \epsilon)\rightarrow M$ via $\Phi(x, t)= \exp_x(t \hat N(x))$ and the restriction $\Phi|_{S\times (-\epsilon, \epsilon)}= \Psi$. We can compute the differential of $\Psi$  in $(p,0)$ and see 
it is injective. 
Hence, we can choose a  neighborhood $O'\cap \Sigma$ of $p$ in $\Sigma$ and $\delta$ such that $\Psi|_{O\times (-\delta, \delta)}$ is a diffeomorphism to its image where $O= S\cap O'$.  By the definition of the map $\Psi$ the image $\Psi(O\times (-\delta, \delta))=:\mathcal H$ is a null hypersurface in $M$ that contains $O$. 

Because $N(q)$ is a null vector in $T\mathcal H$ we also have that $N(q)\in T_q\mathcal H^{\perp}$ for $q\in S$. Now, we extend $N$ to a smooth, non-vanishing null vector field  $N$ on a neighborhood of $p\in \mathcal H$ in $\mathcal H$. Therefore the second fundamental form of $\mathcal H$ in $p$ can be computed as $\langle \nabla_vN(p), v\rangle = \Pi^M_{\mathcal H}(v,v)$ and so 
$$\langle \nabla_vN(p), v\rangle = \Pi^M_{\mathcal H}(v,v)= \Pi^{\Sigma}_S(v,v)= \lambda \langle v, v\rangle \ \forall v\in S.$$

\begin{remark}
Let us make a few comments at this point. If $K$ is a nonvanishing, smooth null vector field along $\mathcal H$, the null Weingarten map is defined as $\nabla_X K \mod K= b(\overline X)$ where $\overline X$ is the equivalence class of spacelike vector $X$ modulo $K$. If $\tilde K= f K$ is another such null vector field for a smooth function $f$, then $\nabla v \tilde K= f \nabla_v K \mod K$. Hence the Weingarten map $b$  is a tensor field and depends at a given point $p\in \mathcal H$ only on the value of $K$ at $p$. 
The null second fundamental form of $\mathcal H$ is then defined as $\overline \Pi_{\mathcal H}^M(\overline X, \overline Y)= \langle \nabla_XK, Y\rangle=\Pi^M_{\mathcal H}(X,Y)$ for $X, Y$ spacelike.  This is consistent with the above.
\end{remark}

We set $\Sigma_r= \Phi( \Sigma\times \{r\})$ as well as $S_r= \Psi(S\times\{r\})$ for $r>0$ small.  $\Sigma_r$ and $S_r$ are space-like.  
A transport relation is given by 
$R_{\mathcal H}= \{ (x,y)\in \mathcal H^2: x\in S, y\in \Psi(x,r)= \exp_x(rN(x))\}$.  Now we choose $\mu_0$ supported in $S$ and consider $\mu_1= (\Psi(\cdot, r))_{\#}$. The coupling $\pi= (\mbox{id}_S, \Psi(\cdot, r))_{\#}\mu_0$ is a null coupling between $\mu_0$ and $\mu_1$. Let $\mu_t$ be the induced null displacement interpolation.  We also set $\gamma_q(t) = \Psi(q,t)$. Then $\gamma_q'(0)=v$.  A transport map $T$ like in Lemma is given by $\Psi(, r)|_U$ for any open subset $U\subset S$ that contains the support of $\mu_0$. Moreover $\tilde K= rN$.
We assume $\mu_r$ is supported in $B_\eta(p)\subset U$ where $B_\eta(p)$ is a  geodesic ball w.r.t. the induced intrinsic distance of $g$ on $S$.

Repeating the calculation as in Theorem \ref{mainforw} we obtain
$$\mbox{tr}(U_x)' + \mbox{tr}(U_x^2) +  \ric_x(\gamma'_x,\gamma_x')=0 \ \ \forall x\in U$$
where $U_x= A_x'A_x^{-1}$ and  $A_x(t)= (DT_t)|_x$, $t\in [0,1]$, with $A_x'(0)w = \nabla_w \tilde K= r\nabla_w N$.  Hence 
$\langle A_p'(0) w,w\rangle = \langle U_p(0)w, w\rangle= r\lambda \langle w, w\rangle$. Here $w\in T_pS$. 
\\
\\
\noindent
{\bf 2.} 
First we prove (1).
For this we choose $\lambda =0$. Recall $\gamma_p'(0)=v$.

By continuity of  $\ric$ and $(x,t)\mapsto \gamma_x'(t)$  we can choose $r$ and $\eta$ such that 
$$\ric(\gamma_x'(t), \gamma_x'(t))\leq - 5 \varepsilon  <0\ \ \  \forall x\in B_\eta(p) \forall t\in [0,1].$$
By continuity of $(x,t) \mapsto U_x(t)$ we can choose  $\eta$ and $r$ small enough to obtain
\begin{align*} \varepsilon\geq \mbox{tr}U^2_x(t) \ \ \& \  \ \frac{1}{n-1}\left( \mbox{tr} U_x(t)\right)^2 \geq -\varepsilon\ \ \  \forall x\in B_\eta(p) \forall t\in [0,1].
\end{align*}
Then we compute
\begin{align*}
\mbox{tr}(U_x)' =- \mbox{tr}(U_x^2) - \ric_x(\gamma'_x,\gamma_x')\geq -\varepsilon+ 5 \varepsilon \geq 4\varepsilon \geq - \frac{1}{n-1} \left(\mbox{tr} U_x(t)\right)^2 +3 \varepsilon .
\end{align*}
Hence, setting $y_x(t) = \det A_x(t)$ we get 
$$(\log y_x)''(t) + \frac{1}{n-1} ((\log y_x)'(t))^2 > 0 \ \ \forall x\in B_\eta(p)$$
and therefore 
$$\left( (y_x)^{\frac{1}{n-1}}\right)'' > 0 \mbox{ on } [0,1] \ \ \forall x\in B_\eta(p).$$
It follows 
$$(y_x)^{\frac{1}{n-1}}(t) < (1-t) (y_x)^{\frac{1}{n-1}}(0) + t (y_x)^{\frac{1}{n-1}}(1) \ \ \forall x\in B_\eta(p) \ \forall t\in [0,1].$$
Here we used that strict positivity of the second derivative is sufficient for strict convexity. 

We can proceed as in the end of the proof of Theorem \ref{mainforw} and obtain 
$$S_{n-1}(\mu_t| \m_{\mathcal H})>(1-t) S_{n-1}(\mu_0|\m_{\mathcal H}) + t S_{n-1}(\mu_1|\m_{\mathcal H}).$$
This contradicts null displacement convexity.
\\
\\
{\bf 3.} Finally we treat (2). 

For this we choose $\lambda=-\frac{1}{N-n+1}\langle \nabla V(p), v\rangle=-\frac{1}{N-n+1}\langle \nabla V(p), N(p)\rangle$. In particular 
$r\lambda= -\frac{1}{N-n+1}\langle \nabla V(p), \tilde K(p)\rangle$ since $rN(p) = \tilde K(p)$.

Provided $\eta>0$ and $r>0$ are small enough, 
by continuity  of $\ric^V$ and $(x,t)\mapsto \gamma_x(t)$  we have 
\begin{itemize}
\item[(1)]
${\displaystyle \ric^V(\gamma_x'(t), \gamma_x'(t))- \frac{1}{N-n+1}\langle \nabla V\circ \gamma_x(t), \gamma_x'(t)\rangle^2\leq - 5 \varepsilon}$
\smallskip\\
$\forall x\in B_\eta(p)\  \forall t\in [0,1] $
\end{itemize}
and 
by continuity of $(x,t)\mapsto A_x(t)$ and $(x,t)\mapsto \langle \nabla V\circ \gamma_x(t), \gamma_x'(t)\rangle $ we have 
\begin{itemize}
\item[(2)] ${\displaystyle 
\left| \mbox{tr} U_x(t) - \langle \nabla V\circ \gamma_x(t), \gamma_x'(t)\rangle\right| \leq C \ \  \forall x\in B_\eta(p)\  \forall t\in [0,1] }$
\end{itemize}
for some constant $C>0$. 

Now we choose $\epsilon' \leq \min\{ \frac{N}{2C} \epsilon, \frac{N(N-n)}{n},  N\}$ and 
since $(x,t)\rightarrow U_x(t)$ and $(x,t)\mapsto \langle \nabla V\circ \gamma_x(t), \gamma_x'(t)\rangle $ are continuous, we can choose $\eta$ and $r$ even smaller such that
\smallskip
\begin{itemize}
\item[(3)]  
${\displaystyle \mbox{tr}U^2_x(t) \leq \mbox{tr}U^2_x(0)+\varepsilon', }$
\medskip
\item[(4)] ${\displaystyle \left| \mbox{tr}U_x(0)-\mbox{tr}U_x(t)\right|\leq \varepsilon'} $
\medskip
\item[(5)] ${\displaystyle \left|
\langle \nabla V(p), v\rangle -\langle \nabla V\circ \gamma_x(t), \gamma'_x(t)\rangle \right| \leq   \varepsilon'.}$
\end{itemize}
$\forall x\in B_\eta(p)\  \forall t\in [0,1]$ 
where, again, \begin{align*}
 (n-1) r^2 \lambda^2 = \mbox{tr}U^2_x(0)= (\mbox{tr}U_x(0))^2 \ \mbox{ and }\ (N-n+1) r \lambda = -  \langle \nabla V(p), \tilde K(p)\rangle.
\end{align*}
\noindent
With this 
we compute first, using (1) and (3),
\begin{align*}
&\mbox{tr}(U_x)'(t) -( V\circ \gamma_x)''(t)
=- \mbox{tr}(U_x^2)(t) - \ric^V_x(\gamma'_x(t),\gamma_x'(t))\\
&\geq - \frac{1}{n-1} \left( (n-1)^2 r^2 \lambda^2\right) - \frac{1}{N-n+1} \left( \langle \nabla V\circ \gamma_x(t), \gamma'_x(t)\rangle \right)^2 + 4\varepsilon  = (\star)
\end{align*}
Next we use
\begin{itemize}
\item[(6)] ${\displaystyle \frac{1}{N}(a+b)^2 + \frac{n}{N(N-n)} \left( b- a\frac{N-n}{n}\right)^2 = \frac{1}{n} a^2 + \frac{1}{N-n} b^2.}$
\end{itemize}
and obtain 
\begin{align*}
(\star)\overset{(6)}{=}& - \frac{1}{N} \left( (n-1) r \lambda - \langle \nabla V\circ \gamma_x(t), \gamma'_x(t)\rangle\right)^2 \\
&- \frac{n}{N(N-n)} \left( \langle \nabla V\circ \gamma_x(t), \gamma_x'(t)\rangle+ (N-n+1) r \lambda\right)^2+ 4 \varepsilon\\
\overset{(4)+(5)}{\geq}& - \frac{1}{N} \left( \mbox{tr} U_x(t) - \langle \nabla V\circ \gamma_x(t), \gamma_x'(t)\rangle \pm \varepsilon'\right)^2 -\underbrace{ \frac{n}{N(N-n)}( \varepsilon')^2}_{\leq \epsilon}+ 4 \varepsilon \\
=& - \frac{1}{N}  \left( \mbox{tr} U_x(t) - \langle \nabla V\circ \gamma_x(t), \gamma_x'(t)\rangle \right)^2\\
&-\underbrace{\frac{2}{N} \left( \mbox{tr} U_x(t) - \langle \nabla V\circ \gamma_x(t), \gamma_x'(t)\rangle \right)\varepsilon'}_{\overset{(2)} {\leq}\frac{2}{N} C \epsilon'\leq \epsilon}- \underbrace{ \frac{1}{N} (\varepsilon')^2}_{\leq \epsilon} + 3 \varepsilon \\
=& - \frac{1}{N}  \left( \mbox{tr} U_x(t) - (V\circ \gamma_x)'(t)\right)^2+\varepsilon
\end{align*}
Consequently 
$$\mbox{tr}(U_x)' -( V\circ \gamma_x)''+\frac{1}{N}  \left( \mbox{tr} U_x(t) - (V\circ \gamma_x)'(t)\right)^2 > 0.$$
and if we set $\mbox{tr}U_x-V\circ \gamma_x= z_x$, it follows
$$
\left(({z_x})^{\frac{1}{N}}\right)'' > 0.$$
Exactly as before it follows 
$$
S_N(\mu_t|\m_{\mathcal H}) > (1-t) S_N(\mu_0|\m_{\mathcal H}) + t S_N(\mu_1|\m_{\mathcal H})
$$
that is a contradiction.
\end{proof}
\begin{definition}\label{def:syn}
A weighted space-time $(M^{n+1},g, e^{-V})$  with $V\in C^0(M)$ satisfies the synthetic $N$-null energy condition for some $N\geq n-1$ if  for every null hypersurface $\mathcal H$ and for every $\mu_0, \mu_1\in \mathcal P(M, \m_{\mathcal H})$ that are null connected via a null coupling $\pi$, it holds 
$$S_{N'}(\mu_t|\m_{\mathcal H})\leq (1-t) S_{N'}(\mu_0|\m_{\mathcal H}) + t S_{N'}(\mu_1|\m_{\mathcal H}) \ \forall N'>N$$
where $\mu_t$ is the corresponding null displacement interpolation. 
\end{definition}
\subsection{Localisation}
\begin{proposition} 
Consider a weighted, time-oriented Lorentzian manifold $(M^{n+1},g, e^{-V})$. Then  the synthetic $N$-null energy condition for  $N\geq n-1$ holds if and only if for every $\mu_0, \mu_1\in \mathcal P(M, \m_{\mathcal H})$ that are acausal and null connected via  a  dynamical null coupling $\Pi$ supported on a null hypersurface $\mathcal H$ one has for every $ t\in [0,1]$ that
\begin{align}\label{local}\rho_t^{-\frac{1}{N}}(\tilde \gamma_t)\geq (1-t)\rho_0^{-\frac{1}{N}}(\tilde \gamma_0) + t \rho_1^{-\frac{1}{N}}(\tilde \gamma_1) \mbox{ for } \Pi\mbox{-almost every }\tilde \gamma\in \mathcal G(\mathcal H)\end{align}
where $\rho_t$ is the density of the $\m_{\mathcal H}$-absolutely continuous part of  the measures $\mu_t$ of the corresponding null displacement interpolation.
\end{proposition}
\begin{proof} By H\"older's inequality it is possible to replace $N$ in \eqref{local} with $N'>N$. 
Integrating \eqref{local} w.r.t. the dynamical null coupling $\Pi$ that comes from $\pi$ yields 
$$S_{N'}(\mu_t|\m_{\mathcal H})\leq (1-t) S_{N'}(\mu_0|\m_{\mathcal H}) + t S_{N'}(\mu_1|\m_{\mathcal H}) \ \forall N'>N.$$

Let us now assume the synthetic $N$-null energy condition. Let $(\mu_t)_{t\in [0,1]}$  be a null displacement interpolation between two acausal measures $\mu_0, \mu_1\in \mathcal P(\mathcal H, \m_{\mathcal H})$ supported on $S_0$ and $S_1$.  $\mu_t$ is induced by a $C^1$ diffeomorphism $T_t: S_0\rightarrow \mathcal H$ and $T_t(S_0)=S_1$ is a spacelike $C^1$ submanifold.  We fix $\tau\in [0,1]$ and $x\in S_\tau$, and we define
\begin{align*}
\Gamma= \{ \tilde \gamma\in \mathcal G(\mathcal H): \tilde \gamma(t)\in B_{\delta}(x)\subset S_\tau\}
\end{align*}
where $B_\delta(x)$ is the ball of radius $\delta>0$ in $S_t$. Assume $\Pi(\Gamma)>0$ and set $\Pi'= \frac{1}{\Pi(\Gamma)} \Pi|_{\Gamma}$.  In particular $\mu_\tau(B_\delta(x))= \Pi(\Gamma). $$(e_t)_{\#}\Pi'= \mu_t'$ is still a null displacement interpolation that is  induced by the same family  of maps $T_t$. Moroever, for $B\subset S_t$
$$\mu_t'(B)= \Pi'(e_t^{-1}(B))= \frac{1}{\Pi(\Gamma)} \Pi|_\Gamma(e_t^{-1}(B))\leq \frac{1}{\Pi(\Gamma)} \Pi(e_t^{-1}(B))= \frac{1}{\Pi(\Gamma)} \mu_t(B).$$
Hence 
$\rho_t'\leq \frac{1}{\Pi(\Gamma)}\rho_t$ and $\rho_\tau'= \frac{1}{\Pi(\Gamma)}\rho_t|_{B_\delta(x)}.$ 
The synthetic $N$-null energy condition yields
\begin{align*}
&\m_{\mathcal H}(B_\delta(x))^{\frac{1}{N'}}\geq \int (\rho_\tau')^{-\frac{1}{N'}} d\mu'_\tau\\
&\geq \int \left[(1-t) (\rho_0')^{-\frac{1}{N'}}(e_0\circ \Lambda_\tau(x)) + t (\rho_1')^{-\frac{1}{N}}(e_1\circ \Lambda_\tau(x))\right] d\mu'_\tau(x)\\
&\geq \Pi(\Gamma)^{\frac{1}{N'}-1}\int_{B_\delta(x)}\left[(1-t) \rho_0^{-\frac{1}{N'}}(e_0\circ \Lambda_\tau(x)) + t \rho_1^{-\frac{1}{N}}(e_1\circ \Lambda_\tau(x))\right] d\mu_\tau(x)
\end{align*}
where the first inequality follows from Jensen's inequality.
Now let $x$ be a  density point of the measure $\mu_t$ w.r.t. $\m_{\mathcal H}$ and also  a density point of $(1-t) \rho_0^{-\frac{1}{N'}}(e_0\circ \Lambda_\tau(x)) + t \rho_1^{-\frac{1}{N}}(e_t\circ \Lambda_\tau(x))=:F(x)$ w.r.t. $\mu_\tau$.  Here we use that $F(x)$ is in $L^1(\mu_\tau)$. We divide the previous inequality by $\mu_\tau(B_\delta(x))^{\frac{1}{N'}}$. Then we take $\delta \downarrow 0$. It follows
$$
\rho_\tau^{\frac{1}{N'}}(x)\geq (1-t) \rho_0^{-\frac{1}{N'}}(e_0\circ \Lambda_\tau(x)) + t \rho_1^{-\frac{1}{N'}}(e_t\circ \Lambda_\tau(x)) 
$$ for $\mu_\tau$-almost every $x\in S_{\tau}$.
Since $\Pi= (\Lambda_\tau)_{\#}\mu_\tau$, our claim follows. 
\end{proof}
\begin{remark}
From Lemma \ref{lem:ma} we know that 
$$\rho_t(\gamma_t(x))= \rho_t(T_t(x))= \rho_0(x)\det DT_t(x) e^{V(x)-V\circ T_t(x)}.$$
Hence $t\in[0,1]\mapsto \rho_t(\gamma_t(x))$ is continuous. Then the previous proposition implies that $\mu_0$-almost every $x\in S_0$ we have that $$\rho_t(\gamma_t)^{-\frac{1}{N}}\geq(1-t) \rho_0(\gamma_0)^{-\frac{1}{N}}+ t\rho_1(\gamma_1)^{-\frac{1}{N}}.$$
and by replacing $T_t$ with $T_{(1-s)t_0+st_1}$ for $t_0, t_1\in [0,1]$ one can conclude that $\left(\rho_t(\gamma_t)^{-\frac{1}{N}}\right)''\leq 0$ in distributional sense on $(0,1)$.
\end{remark}
\section{Applications}
\subsection{Hawking area monotonicity}
\begin{definition} 
Let $\mathcal H$ be a null hypersurface and let $K: \mathcal H\rightarrow T\mathcal H$ be the normal $C^1$ null vector field as in Subsection \ref{subsec:null}.  We call $\mathcal H$  future null complete, if $\exp_x(tK(x))\in \mathcal H$ for all $t\in[0, \infty)$ and for all $x\in \mathcal H$. 

The Lorentzian manifold $(M,g)$ is called future null complete if $\exp(tv)$ is defined for all $t\in [0,\infty)$ for every future-directed null vector.
\end{definition}
\begin{remark}
Let $\mathcal H$ be future null complete and $S\subset \mathcal H$  an acausal spacelike submanifold. Let $\tilde K$ be a $C^1$ null vectorfield along $S$. Then $T_t: S\rightarrow \mathcal H$ with $T_t(x) = \exp_x(t\tilde K(x))$ is a $C^1$ diffeomorphism for all $t\in [0, \infty)$ and for a probability measure $\mu$ on $S$, $(T_t)_{\#}\mu= \mu_t$ is a  probability measure on $S_t$ for all $t\in [0, \infty)$. 
\end{remark}
\begin{theorem}[Hawking Monotonicity] Let $(M^{n+1},g, e^{-V})$ be a  weighted space-time that satisfies the synthetic $N$-null energy condition for $N\geq n-1$ and let $\mathcal H\subset M$ be a null hypersurface.
Assume $\mathcal H$ is future null complete. Let $\Sigma_0, \Sigma_1\subset M$ two acausal spacelike hypersurfaces. Define $\Sigma_i\cap \mathcal H= S_i$, $i=0,1$, and assume $S_0\subset J^-(S_1)$.  Then 
$$\m_{\mathcal H}(S_0)\leq \m_{\mathcal H}(S_1).$$
\end{theorem}
\begin{proof} 
Assume there exist $S_0$ and $S_1$ such that 
$$\m_{\mathcal H} (S_0) > \m_{\mathcal H}(S_1).$$
By Example \ref{example} there exists a map $T: S_0\rightarrow S_1$ given by $T(x)= \exp_x(\tilde K(x))$ for a $C^1$ null vector field along $S_0$. Since $\mathcal H$ is future complete, the family of maps $T_t(x)= \exp_x(t \tilde K(x))$ are $C^1$ diffeomorphisms for all $t\in (0, \infty)$ \cite[Lemma 4.15]{chdegaho} and $T_t(S_0)= S_t$ are spacelike and acausal hypersurfaces  in $\mathcal H$ of codimension $2$ in $M$. 

Let $N'>N$. We define $\mu_0= \frac{1}{\m_{\mathcal H}(S_0)}\m_{\mathcal H}|_{S_0}$ and $(T_t)_{\#}\mu_0=:\mu_t$. It follows  that 
$$-S_{N'}(\mu_0|\m_{\mathcal H})= \m_{\mathcal H} (S_0)^{\frac{1}{N'}}$$
and by Jensen inequality also 
$$-S_{N'}(\mu_1|\m_{\mathcal H}) \leq \m_{\mathcal H}(\supp \mu_1)^{\frac{1}{N'}}\leq \m_{\mathcal H}(S_1)^{\frac{1}{N'}}.$$
Hence 
$$-S_{N'}(\mu_0|\m_{\mathcal H})> - S_{N'}(\mu_1|\m_{\mathcal H}).$$
Now, pick $t_0\gg 1$ and let $\hat T_{\tau}(x)= \exp_x(\tau t_0 \tilde K(x))$ and define $\nu_{\tau}:= (\hat T_{\tau})_{\#} (\mu_0)$ for $\tau\in [0,1]$.  For $\tau_1=\frac{1}{t_0}$ we have $\nu_{\tau_1}= \mu_1$. 

It follows
\begin{align*}
-S_{N'}(\mu_0|\m_{\mathcal H})&> -S_{N'}(\nu_{\tau_1}|\m_{\mathcal H})\\
& \geq - (1-\tau_1) S_{N'}(\nu_0|\m_{\mathcal H}) - \tau_1 S_{N'}(\nu_1|\m_{\mathcal H}).
\end{align*}
For $t_0\gg 1$ large enough we have a contradiction. 
\end{proof}
\subsection{Penrose singularity theorem}
\begin{definition}\label{def:fuco}
Let  $(M^{n+1}, g, e^{-V})$ be a weighted space-time and let  $\mathcal H\subset M$ be a null hypersurface. Let $\Sigma\subset M$ be a acausal spacelike hypersurface and we set $\Sigma\cap \mathcal H= S$.  For every $p\in S$ there exists a geodesic ball $B_\eta(p)\subset S$ w.r.t. $g|_S$ and a null vectorfield $\tilde K: S\rightarrow T\mathcal H$  such that $x\in S\mapsto T_t(x)= \exp_x(\tilde K(x))$ is a $C^1$ diffeomorphism for all $t\in (0,t_p)$ and for some $t_p>0$.
\smallskip\\
We say $S$ is  {\it synthetically future converging in $\mathcal H$} if for every $p\in S$ there exists $\eta>0$, $T_t$ as above and $\epsilon(p)\in (0,1)$ s.t. 
$$
\underline{\frac{d}{dt}}\log \m_{\mathcal H}(T_t(B_\delta (p)))|_{t=0} \leq  -\epsilon(p)<0
$$
for any  $\delta\in (0, \eta)$ where $\underline{\frac{d}{dt}}f(t):= \liminf_{h\downarrow 0} \frac{f(t+h)-f(t)}{h}$. 

\end{definition}
\begin{remark}
This definition of {\it  synthetically future converging in $\mathcal H$} is motivated by related definitions of  mean curvature lower bounds for the boundary of a subset in a metric measure space that satisfies a synthetic Ricci curvature bound or for measured Lorentzian length spaces (see \cite{kettererhk, ketterermean, bukemcwo, camoreview, camolorentz}).
\end{remark}

\begin{lemma} 
Let  $(M^{n+1}, g, e^{-V})$ be a weighted space-time with $V\in C^{\infty}(M)$ and let  $\mathcal H\subset M$ be a null hypersurface. Let $\Sigma\subset M$ be an acausal and spacelike hypersurface.  $\Sigma\cap \mathcal H= S$ is synthetically future converging in $\mathcal H$ if and only if 
$$ H_{V, w}(p):= H_w(p)+\langle \nabla V(p), w\rangle>0, \ p\in S$$
and for every future-directed null vector $w\in T_p\mathcal H$ where $H_w(p)= \langle {\bf H}(p), w\rangle$ and ${\bf H}$ is the mean curvature vector of $S$. 
We call $H_{V, w}$ the weighted mean curvature in direction of $w$. 
\end{lemma}
\begin{proof} We assume first that $S$ is future converging in $\mathcal H$ like in Definition \ref{def:fuco}. 
By the area formula we have 
$$
\m_{\mathcal H} (T_t(B_\delta(p)))= \int_{B_\delta(p)} e^{-V\circ T_t(x)} \det DT_t(x) d\vol_{\mathcal H}(x).
$$
Differentiating this  formula at $t=0$ yields 
\begin{align*}
&\frac{d}{dt}\Big|_{t=0}\m_{\mathcal H} (T_t(B_\delta(p)))= \int_{B_\delta(p)} 
\frac{d}{dt}\Big|_{t=0}e^{-V\circ T_t(x)} \det DT_t(x) d\vol_{\mathcal H}(x)\\
&= \int_{B_\delta(p)} \left(\mbox{tr} \frac{d}{dt}|_{t=0} DT_t(x) - \langle \nabla V(x), \tilde K(x)\rangle \right)d\m_{\mathcal H}(x).
\end{align*}
Moreover 
 \begin{align*}\mbox{tr} \frac{d}{dt}\Big|_{t=0} DT_t(x) = \mbox{tr} D \tilde K(x)&= \sum_{i=1}^{n-1} \langle \nabla e_i \tilde K (x), e_i\rangle\\
&= - \langle {\bf H}(x), \tilde K(x)\rangle=-H_{\tilde K(x)}(x)
\end{align*} for alle $x\in B_\delta(p)$ where $e_1, \dots, e_{n-1}$ is  an orthonormal  bases in $T_xS$.

Hence 
\begin{align*}
0>-\epsilon(p)&\geq \frac{d^-}{dt}\log \m_{\mathcal H}(T_t(B_\delta (p)))|_{t=0}\\
&= \frac{1}{\m_{\mathcal H}(B_\delta(p))} \frac{d}{dt} \m_{\mathcal H}(T_t(B_\delta (p)))|_{t=0}\\
& = \frac{1}{\m_{\mathcal H}(B_\delta(p))} \int_{B_\delta(p)} \left( -H_{\tilde K(x)}(x)- \langle \nabla V(x), \tilde K(x)\rangle \right) d\m_{\mathcal H}(x).
\end{align*}
If $\delta\downarrow 0$, it follows 
\begin{align}\label{meancurv}
H_{\tilde K(p)}(p)+ \langle \nabla V(p), \tilde K(p)\rangle>0.
\end{align}
For every future-directed null vector $w\in T_p\mathcal H$ there exists $\lambda\in \R$ such that $\lambda \tilde K(p)=w$. It follows that $H_{V,w}(p)>0$. 

If we assume \eqref{meancurv}, then for every $p\in S$ we find $\eta>0$ and $\epsilon(p)>0$, such that 
$$
H_{\tilde K(x)}(x)+ \langle \nabla V(x), \tilde K(x)\rangle\geq \epsilon(p) \ \mbox{ for } x\in B_\delta(p) \mbox{ and  for all } \delta\in (0, \eta)
$$
where $\tilde K$ is the vector field that we find according to Definition \ref{def:fuco}.
With the previous computations we see that $S$ is future convergin in $\mathcal H$. 
\end{proof}

\begin{remark} 
The  definition of {\it future converging}   for a codimension two submanifold  $S$ in a Lorentzian manifold $(M,g)$ is that $\langle {\bf H}, v\rangle>0$ for every future directed null vector $v$ normal to $S$ \cite[Chapter 14]{oneillsemi}. 
 \end{remark}
For a weighted space-time  we make the following definition.
\begin{definition}\label{def:trapped}
Let $(M,g,e^{-V}\vol_g)$ be a weighted Lorentzian manifold with $V\in C^{\infty}(M)$.
We call a codimension two submanifold $S\subset M$    Bakry-Emery future converging (or Bakry-Emery trapped) if 
\begin{align*}
\langle {\bf H}, v\rangle + \langle \nabla V, v\rangle >0 
\end{align*}
for every future-directed normal null vector at $S$. 

If $V\in C^0(M)$, consider $\mathcal C$, $\underline {\mathcal C}$, $\mathcal H$, $\underline {\mathcal H}$, $L$ and $\underline{L}$ as in Example \ref{ex:nullcones}.  We call the  codimension two submanifold $S\subset M$    synthetically  future converging (or synthetically trapped) if $S$ is future converging in $\mathcal H$ and future converging in ${\underline{\mathcal H}}$. 
\end{definition}
\begin{corollary}\label{cor:penrose} Consider a weighted space-time  $(M,g,e^{-V})$ that satisfies the synthetic $N$-null energy condition and let $\mathcal H$ be a null hypersurface.
Let $\Sigma\subset M$ be an acausal spacelike hypersurfaces and define $\Sigma\cap \mathcal H= S$. Assume $S$ is future converging in $\mathcal H$. Then $\mathcal H$ is not future null complete. 
\end{corollary}
\begin{proof}  For $p\in S$ we choose $\eta>0$ and $\tilde K$ as in Definition \ref{def:fuco}. We argue by contradiction. Assume $\mathcal H$ is future complete. It follows that $T_t(x) = \exp_x(t \tilde K(x)), x\in B_\eta(p),$ is a $C^1$ diffeomorphism for all $t\in (0,\infty)$.  
We define $\mu_0= \frac{1}{\m_{\mathcal H}}\m_{\mathcal H}|_{B_\delta(p)}$ and $\mu_t= (T_t)_{\#}\mu_0$. 
Since $\mathcal H$ is future complete, it follows, similarly as in the Hawking Monotonicity Theorem, that 
$$S_{N'}(\mu_0|\m_{\mathcal H})\geq S_{N'}(\mu_1|\m_{\mathcal H}).$$
for all $N'>N$. Then  it follows with the synthetic null energy condition
\begin{align*}
0&>-(N'-1) \m_{\mathcal H}(B_\delta(p))^{\frac{1}{N'}} \epsilon(p)\\
&\geq\liminf_{\tau\downarrow 0} \frac{1}{\tau} \left( \m_{\mathcal H}(\hat T_{\tau}B_\delta(p))^{\frac{1}{N'}} - \m_{\mathcal H}(B_\delta(p))^{\frac{1}{N'}}\right)\\
\ \ \ \ \ \ &\geq -S_{N'}(\nu_1|\m_{\mathcal H})+ S_{N'}(\nu_0|\m_{\mathcal H})\geq 0.
\end{align*}
This is a contradiction.
\end{proof}
\begin{remark}
We can refine the analysis of the previous proof as follows.

Let $p\in S$ and $\eta>0$ and $\tilde K$ as above such that $T_t(x) = \exp_x(t \tilde K(x)), x\in B_\eta(p),$ is a $C^1$ diffeomorphism for all $t\in (0,t_p)$ and for some $t_p>0$. 
We define $\mu_0= \frac{1}{\m_{\mathcal H}}\m_{\mathcal H}|_{B_\delta(p)}$ for $\delta\in (0, \eta)$ and $\mu_t= (T_t)_{\#}\mu_0$. 

Recall that area formula implies 
$$\m_{\mathcal H}(T_t(B_\delta(p))= \int_{B_\delta(p)} \det DT_t e^{-V\circ T_t} d\vol_{\mathcal H}= \int_{B_\delta(p)} \frac{\rho_0}{\rho_t\circ T_t} d\m_{\mathcal H}.$$
Hence  future converging in $\mathcal H$ yields
\begin{align*}-\epsilon(p)&\geq {\m_{\mathcal H}(B_\delta(p))}^{-1}\frac{d}{dt}\Big|_{t=0} \m_{\mathcal H}(T_t(B_\delta(p))\\
&\geq {\m_{\mathcal H}{(B_\delta(p))}^{-1}}\int_{B_\delta(p)} \frac{d}{dt} \Big|_{t=0}\frac{\rho_0}{\rho_t\circ T_t} d\m_{\mathcal H}.\end{align*}
and with $\delta\downarrow 0$ one gets
$$0>-\epsilon(p)\geq\frac{d}{dt}\Big|_{t=0}\frac{\rho_0(p)}{\rho_t\circ T_t(p)}=\frac{d}{dt}\Big|_{t=0}\log(\rho_t\circ T_t(p))^{-1} .$$ 
Since $t\mapsto (\rho_t\circ T_t(p))^{-\frac{1}{N'}}$ is concave,  we have $t_p<\frac{1}{\epsilon(x)}$ by Riccatti comparison. 
\end{remark}
Now, we will prove a Penrose Singularity theorem for the synthetic  $N$-null energy condition. 
\begin{theorem}
Let $(M,g,e^{-V})$ be a weighted space-time that is globally hyperbolic and $V\in C^0(M)$.  Assume there exists a non-compact Cauchy hypersurface $\Sigma\subset M$, $M$ contains a synthetically future converging, compact, oriented codimension two submanifold $S$ and the synthetic $N$-null energy condition for $N\geq n-1$ holds. Then $M$ is future null incomplete.
\end{theorem}
\begin{proof}  Let $\mathcal C$, $\underline {\mathcal C}$, $\mathcal H$, $\underline {\mathcal H}$, $L$ and $\underline{L}$ as in Example \ref{ex:nullcones}. 
We define the map $T_t(x)=\exp_x(tL(x))$ and the map $\underline{T_t}(x)= \exp_x(t\underline L(x))$ on $S$. Assume now that $(M,g)$ is future null complete. Then it follows that for a  $x\in S$ $\exp_x(tL(x))=\gamma_x(t)$ and $\exp_x(t\underline L(x))=\hat \gamma_x(t)$ are defined on $[0,\infty)$. 

Let $p\in S$ and  define $\mu_0= \frac{1}{\m_{\mathcal H}(B_\delta(p))} \m_{\mathcal H}|_{B_\delta(p)}$ Let $t_p>0$ such that $T_t$ is a diffeomorphism on $B_\delta(p)$ for all $t\in (0,t_p)$ and $\mu_t:= (T_t)_{\#}\mu_0$.

Let $\rho_t$ be the density of $\mu_t$ w.r.t. $\m_{\mathcal H}$.  By the previous proof we have that $t_p\in (0, \frac{1}{\epsilon})$. 
Hence, there is a focal point $\tau$ before $\frac{1}{\epsilon}$. For $t>\frac{1}{\epsilon}\geq \tau$ there exists a time-like geodesic from $S$ to $\gamma_x(t)$ and therefore $\gamma_x(t)\notin \partial J^+(S)$. Similarly we argue for $\hat \gamma_x(t)$. 

It follows that $\partial J^+(S)$ is a closed and bounded, hence compact, subset of $\mathcal C\cup \underline{\mathcal C}$.

From here we can follow the proof of the classical Penrose Singularity Theorem. Let us outline the argument. 

In view of the time-orientability of $(M,g)$, there is a global timelike vector field $T$ whose
integral curves are timelike, foliate $M$ and intersect the Cauchy hypersurface $\Sigma$ exactly
once. Since $J^+(S)$  is
a future set, the  topological boundary $\partial J^+(S)$ is an achronal $(n-1)$-dimensional closed
Lipschitz submanifold without boundary. Hence, every integral curve of $T$ intersects $\partial J^+(S)$ at most once.  The projection $P$ of $\partial J^+(S)$ to the Cauchy Hypersurface $\Sigma$ along the flow lines of $T$ is continuous and bijective.
Since $\partial J^+(S)$ is compact,  $P$ is a homeomorphism. This is contradiction, since $\Sigma$ was assumed to be non-compact. 
\end{proof}
\bibliography{new}
\bibliographystyle{amsalpha}
\end{document}